\documentclass[11pt,twoside, final]{amsart}
\copyrightinfo{0}{Iranian Mathematical Society}
\usepackage{amsmath,amsthm,amscd,amsfonts,amssymb,enumerate}
\usepackage{graphicx}		
\usepackage{color}
\usepackage[colorlinks]{hyperref}
 
\newtheorem{theorem}{Theorem}[section]

\newtheorem{corollary}[theorem]{Corollary}
\theoremstyle{definition}

\newtheorem{remark}[theorem]{Remark}
\numberwithin{equation}{section}

\newcommand{\ZZ}{\mathbb Z}

 \commby{Hamid Pezeshk}
\begin{document}

 
\title[Jordan superder and Jordan super-bider of trivial ext]{On Jordan superderivations and Jordan super-biderivations of trivial extensions and triangular matrix rings}
  
\author[Hassan Cheraghpour and Madineh Jafari]{Hassan Cheraghpour$ ^{*} $ and Madineh Jafari} 

\address{University of Primorska, FAMNIT and IAM, Glagolja{\v s}ka 8, 6000 Koper, Slovenia.} 
\email{cheraghpour.hassan@yahoo.com}

\email{madineh.jafari3978@gmail.com}

\thanks{This work is supported in part by the Slovenian Research Agency (project N1-0210).}
\thanks{$ ^{*} $Corresponding Author: Hassan Cheraghpour}
%
\maketitle
%

\begin{abstract}
Triangular matrix rings are example of trivial extensions. In this article we describe the Jordan superderivations of the trivial extensions and upper triangular matrix rings. We deduce then that any Jordan superderivation of an upper triangular matrix ring, under some conditions, is a derivation, and any Jordan super-biderivation of a trivial extension, and hence an upper triangular matrix ring, is a Jordan biderivation.
\end{abstract}
 
\section{\bf Introduction}
\noindent
Study of Jordan derivations have a great history in algebra and analysis. In 1957, Herstein \cite{Her} proved that every Jordan derivation from a 2-torsion free prime ring into itself is a derivation and that there are no nonzero antiderivations on a prime ring. In 1970, Sinclair \cite{Sin} proved that every continuous linear Jordan derivation on semisimple Banach algebras is a derivation. In 1988, Bre\v{s}ar \cite{Bre} showed that every additive Jordan derivation from a $2$-torsion free semiprime ring into itself is a derivation. In 1998, Zhang \cite{Zhan2} proved that every linear Jordan derivation on nest algebras is an inner derivation. In 2003, Fo\v{s}ner \cite{Fosn} extend Herstein's theorem on Jordan derivations of prime rings to superalgebras. In 2005, Benkovi\v{c} \cite{BenJ} determined Jordan derivations on triangular matrices over commutative rings and proved that every Jordan derivation from the algebra of all upper triangular matrices into its arbitrary bimodule is the sum of a derivation and an antiderivation. In 2006, Zhang and Yu \cite{Zhan} showed that every Jordan derivation of triangular algebras is a derivation, so every Jordan derivation from the algebra of all upper triangular matrices into itself is a derivation. In 2007, Li and Lu \cite{Li} proved that every additive Jordan derivation on reflexive algebras is a derivation which generalized Zhang's result. In 2012, Ghosseiri \cite{Gho} described Jordan left derivations and generalized Jordan left derivations of matrix rings. Somewhat later, Ghahramani \cite{Gha} described Jordan derivations of trivial extensions and showed that every Jordan derivation of trivial extension, under some conditions, is the sum of derivation and an antiderivation. Then, in 2016, Hadj Ahmed \cite{Haj} showed that under certain conditions, any Jordan biderivation of an upper triangular matrix ring is a biderivation.

In 1962, Nagata \cite{Nag} introduced the trivial extensions, and Kitamara \cite{Kit} in 1983 investigated quotient rings of trivial extensions. In 1984, Assem et al. \cite{Ass} represented finite trivial extension algebras, and Hughes and Waschbusch \cite{Hug} in 1993 described the trivial extensions of tilted algebras, and then Ghosseiri \cite{Ghos} described derivations and biderivations of trivial extensions and triangular matrix rings.

Let $R$ be a ring with identity, and, $M$ and $N$ be unitary $R$-bimodules. For each $x,y\in R$, denote the commutator of $x,y$ by $[x,y]=xy-yx$. An additive mapping $d:R\rightarrow M$ is said to be a {\it derivation} if $d(ab)=d(a)b+ad(b)$ for all $a,b\in R$. Let $a\in R$. The mapping $I_a:R\rightarrow R$ given by $I_a(x)=[x,a]$ is easily seen to be a derivation of $R$. $I_a$ is called the {\it inner derivation} induced by $a$. For each $x,y \in R$, define the \textit{Jordan product} of $x,y$ by $x \circ y =xy+yx$. In \cite{Bei}, Beidar et al. defined the {\it Jordan homomorphism} as an additive mapping $\varphi:M\rightarrow N$ satisfying $\varphi(m \circ m')= \varphi(m) \circ \varphi(m')$ for any $ m, m' \in M $. Similarly, we define {\it Jordan R-homomorphism} as an additive mapping $\varphi:M\rightarrow N$ satisfying $\varphi(r \circ m)=r  \circ \varphi(m)$ for all $ r \in R $ and $ m \in M $. An additive mapping $d:R\rightarrow M$ is said to be a {\it Jordan derivation} if $d(a \circ b)=d(a) \circ b+a \circ d(b)$ for all $a,b\in R$. Obviously, every derivation is a Jordan derivation, but not conversely. A biadditive mapping $d:R\times R\rightarrow M$ is called a {\it Jordan biderivation} if it is a Jordan derivation in each component; that is, 
\begin{equation*}
d(x \circ y,z)=d(x,z) \circ y+x \circ d(y,z),
\end{equation*} 
and
\begin{equation*}
d(x,y \circ z)=d(x,y) \circ z+y \circ d(x,z)
\end{equation*}
are fulfilled for all $x,y,z$ in $R$.

Let $ A $ be a ring with identity. Following \cite{Ghah}, $A$ is a \textit{superring} if it is a $ \mathbb{Z}_{2}-$graded ring; this means that there exist additive subgroups $ A_{0} $ and $ A_{1} $ of $ A $ such that $ A=A_{0} \oplus A_{1} $, $ A_{0}A_{0}\subseteq A_{0} $ (this means $ A_{0} $ is a subring of $ A $), $ A_{0}A_{1} \subseteq A_{1} $, $A_{1}A_{0} \subseteq A_{1} $ and $ A_{1}A_{1} \subseteq A_{0}$. We say that $ A_{0} $ is the \textit{even} part, and $ A_{1} $ is the \textit{odd} part of $ A $. If $ a \in A_{i} $, $ i=0 $ or $ i=1 $, then we say that $ a $ is \textit{homogeneous} of \textit{degree} $ i $ and write $ |a|=i $. Let $ A $ be a superring with identity, and let $ i=0$ or $i=1 $. We say that an additive mapping $ d_{i}:A \longrightarrow A $ is a \textit{superderivation} of degree $ i $, if it satisfies $ d_{i}(A_{j}) \subseteq A_{i+j} $ (index modulo 2) and 
\begin{center}
$d_{i}(xy)=d_{i}(x)y+(-1)^{i|x|}xd_{i}(y)$ \ \ \ \ \ \ for all $ x,y \in A_{0} \cup A_{1}$.
\end{center}
A \textit{superderivation} of $ A $ is the sum of a superderivation of degree $ 0 $ and a superderivation of degree $ 1 $.

Following \cite{Wan}, for all $x,y \in A_{0} \cup A_{1}$, define the \textit{Jordan superproduct} of $x,y$ by
\begin{equation*}
x \circ_{s} y=xy+(-1)^{|x||y|}yx. 
\end{equation*}
Note that, if either of  $ x $ or $ y $ in $ \in A_{0} $, then $x \circ_{s} y=x \circ y$. Let $ A $ be a superring with identity, and let $ i=0$ or $i=1 $. We say that an additive mapping $ d_{i}:A \longrightarrow A $ is a \textit{Jordan superderivation} of degree $ i $, if it satisfies $ d_{i}(A_{j}) \subseteq A_{i+j} $ (index modulo 2) and 
\begin{center}
$d_{i}(x \circ_{s} y)=d_{i}(x) \circ_{s} y+(-1)^{i|x|} x \circ_{s} d_{i}(y)$ \ \ \ \ \ \ for all $ x,y \in A_{0} \cup A_{1}$.
\end{center}
A \textit{Jordan superderivation} of $ A $ is the sum of a Jordan superderivation of degree $ 0 $ and a Jordan superderivation of degree $ 1 $.

Following \cite{Fos}, we define a \textit{super-biderivation} of $A$ as a biadditive mapping $ B: A \times A\longrightarrow A $ such that for every $ x_{0} \in A_{0}$, the mappings $ x\longmapsto B(x_{0},x) $ and $ x\longmapsto B(x,x_{0})$ are superderivations of degree $ 0 $, and for every $ x_{1} \in A_{1}$ the mappings $ x\longmapsto B(x_{1},x) $ and $ x\longmapsto \sigma (B(x,x_{1}))$ are superderivations of degree $ 1 $, where $ \sigma : A \longrightarrow A $ is the automorphism such that for any $ a_{0} \in A_{0} $ and $ a_{1} \in A_{1} $ is defined by $ \sigma (a_{0} + a_{1})=a_{0} - a_{1} $. Similarly, we define a \textit{Jordan super-biderivation} of $A$ as a biadditive mapping $ B: A \times A\longrightarrow A $ such that for every $ x_{0} \in A_{0}$, the mappings $ x\longmapsto B(x_{0},x) $ and $ x\longmapsto B(x,x_{0})$ are Jordan superderivations of degree $ 0 $, and for every $ x_{1} \in A_{1}$ the mappings $ x\longmapsto B(x_{1},x) $ and $ x\longmapsto \sigma (B(x,x_{1}))$ are Jordan superderivations of degree $ 1 $, where $ \sigma : A \longrightarrow A $ is the automorphism such that for any $ a_{0} \in A_{0} $ and $ a_{1} \in A_{1} $ is defined by $ \sigma (a_{0} + a_{1})=a_{0} - a_{1} $. This means that $ B: A \times A\longrightarrow A $ is Jordan super-biderivation if it satisfies $ B(A_{i},A_{j}) \subseteq A_{i+j} $ (index modulo 2) such that, for any $ x, y, z \in A_{0} \cup A_{1} $ we have: 
\begin{equation}\label{11}
\;B(x,y\circ_{s}z)=B(x,y)\circ_{s}z+(-1)^{|x||y|}y\circ_{s}B(x,z),
\end{equation}
\begin{equation}\label{12}
B(x\circ_{s}y,z)=x\circ_{s}B(y,z)+(-1)^{|y||z|}B(x,z)\circ_{s}y.
\end{equation}

Let $R$ be a ring with identity, and $M$ be a unitary $R$-bimodule. The {\it trivial extension} $T(R,M)$ of $R$ by $M$ 
is defined to be 
\begin{equation*}
T(R,M)=\{(r,m):r\in R,m\in M\}.
\end{equation*}
It is easy to see that $T(R,M)$ with the componentwise addition and the multiplication given by
\begin{equation*}
(r,m)(r',m')=(rr',rm'+mr')\;\;\;\;\;\;(r,r'\in R;m,m'\in M)
\end{equation*}
is a ring with the multiplicative identity $(1,0)$. By letting
\begin{equation*}
A_{0}=\{(r,0):r\in R\},\;A_{1}=\{(0,m):m\in M\},
\end{equation*}
it is easily verified that $ A=T(R,M) $ is a superring.
\\\indent Let $R$ and $S$ be rings with identity, $M$ be a unitary $(R,S)$-bimodule, and $T=\left(\begin{array}{cc}R&M\\0&S\end{array}\right) $ be the upper triangular matrix ring determined by $R,S$ and $M$ with the usual addition and multiplication of matrices. 

Let $R $, $ S $ and $ M $ be as above. Then one can easily verify that $M$ can be made into a unitary $R\times S$-bimodule via the scalar multiplications given by 
\begin{equation}\label{13}
(r,s)m=rm\;\;\;\mbox{and}\;\;\;m(r,s)=ms\;\;\;\;\;((r,s)\in R\times S,m\in M).
\end{equation}
Hence, one observes that $T(R\times S,M)$ is the trivial extension of $R\times S$ by $M$. Now, it is straightforward to show that the mapping
\begin{equation}\label{14}
T\rightarrow T(R\times S,M)\;\;\;\mbox{given by}\;\;\;\left(\begin{array}{cc}r&m\\0&s\end{array}\right)\mapsto ((r,s),m)
\end{equation} 
is a ring isomorphism. Therefore, upon the (natural) identification above, the upper triangular matrix rings are examples of trivial extensions, as claimed above. Since it turns out that the trivial extensions are easier to work with, one can prove a property for an upper triangular matrix ring via proving the same property for the corresponding trivial extension.

In this paper first we determine the structure of Jordan superderivations and Jordan super-biderivations of $T(R,M)$ (Theorems \ref{2.1} and \ref{2.6}), and then, following the procedure explained above, we determine the structure of Jordan superderivations of $T$ (Theorem \ref{2.3}). We deduce then that any Jordan superderivation of an upper triangular matrix ring is just a derivation, and any Jordan super-biderivation of a trivial extension, and hence an upper triangular matrix ring, is a Jordan biderivation. 

In the sequel, all rings are assumed to be unital, and unless there is a doubt of ambiguity, the zero elements of rings and modules, and zero functions are all denoted by $0$. As usual, $E_{ij}$ stands for the standard matrix unit.

\section{\bf Main results and proofs}
Let $R$ be a ring with identity, $M$ be unitary $R$-bimodule and $T(R,M)$ be the trivial extension of $R$ by $M$. Our first result describes the Jordan superderivations of $T(R,M)$.

\begin{theorem}\label{2.1} 
Let $d_{0}$ and $ d_{1} $ be, respectively, Jordan superderivations of degree $ 0 $ and $ 1 $ of the trivial extension $T(R,M)$. Then there exists
\begin{enumerate}[\normalfont (i)]
\item a Jordan derivation $\delta$ of $R$,
\item a Jordan derivation $\gamma :R\rightarrow M$,
\item a bimodule Jordan homomorphism $f:M\rightarrow R$ satisfying 
\begin{equation}\label{21}
m \circ f(m')=f(m) \circ m' \;\;\; (m,m'\in M),
\end{equation}
\item and there exists an additive mapping $g$ on $M$ satisfying 
\begin{equation}\label{22}
g(r \circ m)=r \circ g(m)+\delta (r) \circ m \;\;\; (r \in R, m \in M),
\end{equation}
\end{enumerate}
such that $ d_{0} $ and $ d_{1} $ can be expressed as
\begin{equation*}
d_{0}(r,m)=(\delta (r),g(m)) \;\;\;\mbox{and}\;\;\;d_{1}(r,m)=(f(m),\gamma (r)).
\end{equation*}

Hence, every Jordan superderivation of $T(R,M)$ is of the form
\begin{equation*}
d(r,m)=(\delta (r)+f(m),g(m)+\gamma (r)),
\end{equation*}
where $ \delta$, $f$, $g$ and $ \gamma $ are given above.

\end{theorem}
\begin{proof}
Let $r\in R$ and set $d_{0}(r,0)=(\delta(r),0)$ and $ d_{1}(r,0)=(0,\gamma(r)) $. Since $d_{0}$ and $ d_{1} $ are additive, so are $\delta$ and $\gamma$. For any $r,r'\in R$ we have
\begin{align*}
(\delta(r \circ r'),0)&=d_{0}(r \circ r',0)=d_{0}((r,0) \circ_{s} (r',0))\\&=d_{0}(r,0) \circ_{s} (r',0)+(r,0) \circ_{s} d_{0}(r',0)\\
&=(\delta(r),0) \circ (r',0)+(r,0) \circ (\delta(r'),0)\\&=(\delta(r) \circ r',0)+(r \circ \delta(r'),0)\\
&=(\delta(r) \circ r'+r \circ \delta(r'),0).
\end{align*}
Therefore, $\delta$ is a Jordan derivation of $ R $. A similar argument shows that $\gamma:R \longrightarrow M$ is a Jordan derivation. This proves (i) and (ii). To prove (iii) and (iv), let $m\in M$ be arbitrary and set $d_{0}(0,m)=(0,g(m))$, $d_{1}(0,m)=(f(m),0)$. Obviously, $f$ and $g$ are additive. Let $r\in R$. Using (i) and (ii), we have
\begin{align*}
(0,g(r \circ m))&=d_{0}(0,r \circ m)=d_{0}((r,0) \circ_{s} (0,m))\\&=d_{0}(r,0)\circ_{s}(0,m)+(r,0)\circ_{s}d_{0}(0,m)\\
&=(\delta(r),0)\circ(0,m)+(r,0)\circ(0,g(m))\\
&=(0,\delta(r)\circ m)+(0,r\circ g(m))\\&=(0,r\circ g(m)+\delta(r)\circ m), \; \mbox{and}
\end{align*}
\begin{align*}
(f(r\circ m),0)&=d_{1}(0,r\circ m)=d_{1}((r,0)\circ_{s}(0,m))\\&=d_{1}(r,0)\circ_{s}(0,m)+(r,0)\circ_{s}d_{1}(0,m)\\
&=(0,\gamma(r))\circ(0,m)+(r,0)\circ(f(m),0)\\
&=(0,0)+(r\circ f(m),0)\\&=(r\circ f(m),0).
\end{align*}
Therefore, $g$ satisfies \eqref{22} and $ f $ is a Jordan $ R-$homomorphism. To prove \eqref{21}, let $m,m'$ be in $M$. Applying $d_{1}$ to $(0,m) \circ_{s}(0,m')=(0,0)$, we have
\begin{align*}
(0,0)&=d_{1}((0,m)\circ_{s}(0,m'))\\&=d_{1}(0,m)\circ_{s}(0,m')-(0,m)\circ_{s}d_{1}(0,m')\\
&=(f(m),0)\circ(0,m')-(0,m)\circ(f(m'),0)\\&=(0,f(m)\circ m')-(0,m\circ f(m'))\\&=(0,f(m)\circ m'-m\circ f(m')).
\end{align*}
Since $ d_{0} $ and $d_{1}$ are additive, using the foregoing results, for all $r\in R$, $m\in M$ we have 
\begin{equation*}
d_{0}(r,m)=(\delta (r),g(m)) \;\;\;\mbox{and}\;\;\; d_{1}(r,m)=(f(m),\gamma (r)).
\end{equation*}
\end{proof}

In \cite{Ghah}, Ghahramani et al. posed four questions about superderivations on $ \ZZ_{2}-$graded rings, and answered some of these questions in the context of quaternion rings and the upper triangular matrix rings $ T_{n}(R) $, where $ R $ is a ring with identity and $ n \geqslant 2 $. Their second \textbf{question} is: 
\\

\textit{Under what conditions on a $ \ZZ_{2}-$graded ring, every superderivation is a derivation?}
\\

This question suggests an analogue question on Jordan superderivations of $ \ZZ_{2}-$graded rings: 
\\

\textbf{Question.} \textit{Under what conditions on a $ \ZZ_{2}-$graded ring, every Jordan superderivation is a Jordan derivation?}
\\

The following corollary determines a condition on the $ \ZZ_{2}-$graded ring $ T(R,M) $ under which this ring satisfies the property described in the question above. 

\begin{corollary}\label{2.2}
If the bimodule $ M $ is 2-torsion $($so that $ M $ is also a $ \ZZ_{2}-$vector space$)$, then the mapping $ d_{1}:T(R,M) \longrightarrow T(R,M) $ given by $ d_{1}(r,m)=(f(m),\gamma(r)) $ is a Jordan derivation, so that $ d=d_{0}+ d_{1}$ is a Jordan derivation of $ T(R,M) $.
\end{corollary}
\begin{proof}
One can easily verify that $ d_{1} $ is a Jordan derivation of $ T(R,M) $ iff $ 2(m \circ f(m'))=0 $ for all $ m, m' \in M $, which is guaranteed by hypothesis.
\end{proof}

Now, upon the identification given in \eqref{14} and using Theorem \ref{2.1}, we can determine the structure of the Jordan superderivations of the upper triangular matrix ring $T=\left(\begin{array}{cc}R&M\\0&S\end{array}\right)$.

\begin{theorem}\label{2.3} 
Let $d_{0}$ and $ d_{1} $ be, respectively, Jordan superderivations of degree $ 0 $ and $ 1 $ of the upper triangular matrix ring $T$, and the rings $ R,S $ be 2-torsion free. Then there are Jordan derivations $\delta_1$ of $R$, $\delta_2$ of $S$, an additive mapping $g$ on $M$ such that
\begin{equation}\label{23}
g(rm+ms)=rg(m)+\delta_1(r)m+g(m)s+m\delta_2(s) \;\; ( r\in R,\; m\in M,\; s\in S),
\end{equation}
and there exists an element $m^*\in M$ such that, for every $X=\left(\begin{array}{cc}r&m\\0&s\end{array}\right)\in T$ we have 
$$d_{0}(X)=\left(\begin{array}{cc}\delta_1(r)&g(m)\\0&\delta_2(s)\end{array}\right),$$
$$d_{1}(X)=\left(\begin{array}{cc}0&rm^*-m^*s\\0&0\end{array}\right).$$

In particular, every Jordan superderivation of $T$ is automatically a Jordan derivation.
\end{theorem}
\begin{proof}
Let us identify $T$ with $T(R\times S,M)$. By Theorem 2.1, there are Jordan derivations $\delta$ of $R\times S$ and $\gamma:R\times S\rightarrow M$ such that, for every $(r,s)\in R\times S$, we have$$d_{0}((r,s),0)=(\delta(r,s),0),$$ and $$d_{1}((r,s),0)=(0,\gamma(r,s)).$$
\indent We claim that, there are Jordan derivations $\delta_1$ of $R$ and $\delta_2$ of $S$ such that
\begin{equation}\label{24}
\delta(r,s)=(\delta_1(r),\delta_2(s))\;\;\;\;\mbox{for all}\;\;r\in R,s\in S.
\end{equation}
To see this, let $\delta(1,0)=(a,b)$. Then, from
\begin{align*}
(2a,2b)&=\delta((1,0) \circ (1,0))\\
&=\delta(1,0) \circ(1,0)+(1,0)\circ \delta(1,0)\\
&=(2a,0)+(2a,0)=(4a,0),
\end{align*}
and 2-torsion freeness of $ R $ and $ S $, it follows that $\delta(1,0)=(0,0)$. Similarly, $\delta(0,1)=(0,0)$. Now let $r$ be in $R$ and assume that $\delta(r,0)=(\delta_1(r),\alpha(r))$. Since $\delta(1,0)=(0,0)$, we have 
\begin{align*}
2(\delta_1(r),\alpha(r))&=2\delta(r,0)=\delta(2r,0)\\
&=\delta((r,0) \circ (1,0))\\
&=\delta(r,0) \circ (1,0)+(r,0)\circ \delta(1,0)\\
&=(2\delta_1(r),0).
\end{align*}
Hence $\alpha(r)=0$. Likewise, there exists a mapping $\delta_2$ on $S$ such that for every $s\in S,\delta(0,s)=(0,\delta_2(s))$. Consequently, for every $r\in R,s\in S$, we have 
$$d_{0}(r,s)=d_{0}(r,0)+d_{0}(0,s)=(\delta_1(r),\delta_2(s)).$$

To see that $\delta_1,\delta_2$ are derivations, let $(r',s')$ be also in $R\times S$. Then 
\begin{align*}
(\delta_1(r\circ r'),\delta_2(s\circ s'))&=\delta(r\circ r',s\circ s')=\delta((r,s)\circ (r',s'))\\
&=\delta(r,s)\circ (r',s')+(r,s)\circ \delta(r',s')\\
&=(\delta_1(r),\delta_2(s))\circ (r',s')+(r,s)\circ (\delta_1(r'),\delta_2(s'))\\
&=(\delta_1(r)\circ r'+r\circ \delta(r'),\delta_2(s)\circ s'+s\circ \delta_2(s')).
\end{align*}
Since $\delta$ is additive, so are $\delta_1,\delta_2$. Therefore, $\delta_1,\delta_2$ are Jordan derivations.

Now, we claim that there exists $m^*\in M$ such that 
\begin{equation}\label{25}
\gamma(r,s)=rm^*-m^*s\;\;\;\;\mbox{for all}\;\;r\in R,s\in S.\end{equation}

Since $\gamma$ is a Jordan derivation, and, $ R $ and $ S $ are 2-torsion free rings, we have $\gamma(1,1)=0$. Put $\gamma(1,0)=m^*$. Then, in view of \eqref{12} and noting that $\gamma(0,1)=-\gamma(1,0)=-m^*$, for every $r\in R,s\in S$, we have 
\begin{align*}
\gamma(2r,0)&=\gamma(r,0)+\gamma(r,0)=\gamma((r,0)(1,0)+(1,0)(r,0))\\
&=\gamma(r,0)(1,0)+(r,0)\gamma(1,0)+\gamma(1,0)(r,0)+(1,0)\gamma(r,0)\\
&=(r,0)m^*+\gamma(r,0)=rm^*+\gamma(r,0).
\end{align*}
Then $ \gamma(r,0)=rm^* $, and
\begin{align*}
\gamma(0,2s)&=\gamma(0,s)+\gamma(0,s)=\gamma((0,1)(0,s)+(0,s)(0,1))\\
&=\gamma(0,1)(0,s)+(0,1)\gamma(0,s)+\gamma(0,s)(1,0)+(0,s)\gamma(0,1)\\
&=-m^*(0,s)+\gamma(0,s)=-m^*s+\gamma(0,s).
\end{align*}
Then $ \gamma(0,s)=-m^*s $. Hence, $\gamma(r,s)=\gamma(r,0)+\gamma(0,s)=rm^*-m^*s$, proving \eqref{25}.

Now, by Theorem \ref{2.1}, there exists a bimodule Jordan $ R \times S $-homomorphism $f:M\rightarrow R\times S$ that satisfies condition \eqref{21}, and an additive mapping $g:M\rightarrow M$ that satisfies condition \eqref{22}. First, we prove that $f=0$: let $m$ be in $M$ and assume that $f(m)=(u,v)$. Using \eqref{13} we have
\begin{align*}
(2u,2v)&=2(u,v)=2f(m)=f(m)+f(m)=f ((1,0)m+m(0,1))\\
&=f(m)(0,1)+(1,0)f(m)=(0,1)(u,v)+(u,v)(1,0)\\
&=(u,0)+(0,v)=(u,v).
\end{align*}
Hence, $u=v=0$.
\\\indent To show that $g$ satisfies \eqref{23}, let $r\in R$, $ s \in S $ and $m\in M$ be arbitrary. Noting that here $M$ is an $R\times S$-bimodule, using \eqref{13} again, we have
\begin{align*}
g(rm+ms)&=g((r,s)m+m(r,s))\\
&=(r,s)g(m)+\delta(r,s)m+g(m)(r,s)+m\delta(r,s)\\
&=rg(m)+\delta_1(r)m+g(m)s+m\delta_2(s).
\end{align*}
\indent Finally, by Theorem \ref{2.1}, the identification described in \eqref{14}, and noting that $f=0$, for for every $ X=\left(\begin{array}{cc}r&m\\0&s\end{array}\right) $ in $ T $ we have 

\begin{align*}
d_{0}(X)=d_{0}((r,s),m)&=(\delta(r,s),g(m))\\&=((\delta_1(r),\delta_2(s)),g(m))
\\&=\left(\begin{array}{cc}\delta_1(r)&g(m)\\0&\delta_2(s)\end{array}\right),
\end{align*}
and
\begin{align*}
d_{1}(X)=d_{1}((r,s),m)&=(f(m),\gamma(r,s))\\&=(0,rm^*-m^*s)
\\&=\left(\begin{array}{cc}0&rm^*-m^*s\\0&0\end{array}\right).
\end{align*}

Now, a striaghtforward computation shows the Jordan superderivation $d = d_{0}+d_{1}$ of $T$ is a Jordan derivation.
\end{proof}

Let $ R $, $ S $ and $ M $ be as above. By the above theorem and \cite [Theorem 2.1] {Zhan}, we have the following interesting corollary: 
\begin{corollary}\label{2.4}
It the bimodule $ M $ is faithful as a left $R$-module and as a right $S$-module. Then, any Jordan superderivation of upper triangular matrix ring $ T $ is a derivation.
\end{corollary}

The theorem above has also the following corollary:
\begin{corollary}\label{2.5}
Every Jordan superderivation of degree $ 1 $ of $ T $ is an inner derivation.
\end{corollary}
\begin{proof}
Define $M^*=m^*E_{12}$, and let $I_{M^*}$ be the inner derivation of $T$ induced by $M^*$. Then, for every $X=\left(\begin{array}{cc}r&m\\0&s\end{array}\right)$ in $T$ we have
$$ I_{M^*}(X) =Xm^*E_{12}-m^*E_{12}X=(rm^*-m^*s)E_{12}=d_{1}(X).$$
\end{proof}

Our next aim is to describe the Jordan super-biderivations of the trivial extension $T(R,M)$.
\begin{theorem}\label{2.6}
Let $B$ be a Jordan super-biderivation of the trivial extension $ T(R,M)$. Then there exists 
\begin{enumerate}[\normalfont (i)]
\item a Jordan biderivation $\delta$ of $R$,
\item a biadditive mapping $\beta:R\times M\rightarrow M$ which is a Jordan derivation in the first coordinate, and
\begin{equation*}
\beta(r,m\circ r')=\beta(r,m)\circ r'+m\circ \delta(r,r') 
\end{equation*}
for all $r,r'\in R$ and $m\in M$,
\item a biadditive mapping $\eta:M\times R\rightarrow M$ which is a Jordan derivation in the second coordinate, and
$$\eta(m \circ r',r)=\eta(m,r)\circ r'+m \circ \delta(r',r)$$
for all $r,r'\in R$ and $m\in M$,
\item and there exists a biadditive mapping $f:M\times M\rightarrow R$ which is a bimodule Jordan $ R-$homomorphism in each coordinate such that for all $r,r'\in R$ and $m,m'\in M$ we have
\begin{align*}
B((r,m),(r',m'))&=(\delta(r,r')+f(m,m'),\beta(r,m')+\eta(m,r')).
\end{align*}
\end{enumerate}

$ B $ is a Jordan biderivation of $ T(R,M) $. In paticular, every Jordan super-biderivation of the upper triangular matrix ring $ T $ is a Jordan biderivation.
\end{theorem}
\begin{proof}
(i) Let $r,r'\in R$, and let $B((r,0),(r',0))=(\delta(r,r'),0)$, for some mapping  $ \delta : R \times R \longrightarrow R $. Since $B$ is biadditive, so is $\delta$. Using \eqref{11}, for every $r_1,r_2,r'\in R$ we have
\begin{align*}
(\delta(r_1\circ r_2,r'),0)&=B((r_1\circ r_2,0),(r',0))=B((r_1,0)\circ_{s}(r_2,0),(r',0))\\
&=(r_1,0)\circ_{s}B((r_2,0),(r',0))+B((r_1,0),(r',0))\circ_{s}(r_2,0)\\&=(r_1,0)\circ (\delta(r_2,r'),(r',0))+(\delta(r_1,r'),0)\circ (r_2,0)\\\
&=(r_1\circ \delta(r_2,r')+\delta(r_1,r')\circ r_2,0).
\end{align*}
Therefore, $\delta$ is a Jordan derivation in the first coordinate. Similarily $\delta$ is Jordan derivation in the second coordinate.\\\indent (ii) Let $r\in R,m\in M$ be arbitrary, and set
$$B((r,0),(0,m))=(0,\beta(r,m))\in R\times M,$$
for some mapping $ \beta:R \times M \longrightarrow M $. Biadditivity of $\beta$ follows from $B$. Let $r'$ be also in $R$. Then, from
\begin{align*}
(0,\beta(r\circ r',m))&=B((r\circ r',0),(0,m))=B((r,0)\circ_{s}(r',0),(0,m))\\&=(r,0)\circ_{s}B((r',0),(0,m))+d((r,0),(0,m))\circ_{s}(r',0)\\
&=(r,0)\circ (0,\beta(r',m))+(0,\beta(r,m))\circ (r',0)\\&=(0,r\circ \beta(r',m)+\beta(r,m)\circ r')
\end{align*}
it follows that $\beta$ is a Jordan derivation in the first coordinate. Moreover using \eqref{12}, we have 
\begin{align*}
(0,\beta(r,r'\circ m))&=B((r,0),(0,r'\circ m))=B((r,0),(r',0)\circ_{s}(0,m))\\&=(r',0)\circ_{s}(B(r,0),(0,m))+B((r,0),(r',0))\circ_{s}(0,m)\\
&=(r',0)\circ (0,\beta(r,m))+(\delta(r,r'),0))\circ (0,m)\\&=(0,r'\circ \beta(r,m))+(0,\delta(r,r')\circ m)\\
&=(0,r' \circ \beta(r,m)+\delta(r,r') \circ m).
\end{align*}

(iii) The proofs of the existence of the mapping $\eta:M\times R\rightarrow M$ satisfying the properties given in (iii) is similar to that of $\beta$ in (ii), hence suppressed.

(iv) Let $m,m'\in M$ be arbitrary, and assume that $$B((0,m),(0,m'))=(f(m,m'),0)\in R\times M,$$ where $ f:M \times M \longrightarrow R $ is a mapping determined by $ B $. Obviously, $f$ is biadditive. Let $r$ be in $R$. Then, from
\begin{align*}
(f(r\circ m,m'),0)&=B((0,r\circ m),(0,m'))=B((r,0)\circ_{s}(0,m),(0,m'))\\&=(r,0)\circ_{s}(B(0,m),(0,m'))+B((r,0),(0,m'))\circ_{s}(0,m)\\
&=(r,0)\circ (f(m,m'),0)-(0,\beta(r,m'))\circ (0,m)\\&=(r\circ f(m,m'),0)
\end{align*}
it follows that $f(r\circ m,m')=r \circ f(m,m')$. In similar fashions, we can show that the other equalities in (iv) hold.

The last conclusion follows from biadditivity of $ B $ and (i)-(iv) above.

Finally, it is now striaghtforward to show that $ B $ is a Jordan biderivation of $ T(R,M) $, and the particular case follows from the isomorphism given in \eqref{14}. 
\end{proof}
\begin{remark}\label{2.7}
For any ring $ R $ and $ n \geqslant 2 $, the upper triangular matrix rings $ T_{n}(R) $ are examples of triangular matrix rings $ T=\left(\begin{array}{cc}R&M\\0&S\end{array}\right) $ by the ring isomorphism 
\begin{equation*}
T_{n}(R) \cong  \left(\begin{array}{cc}R&R^{n-1}\\0&T_{n-1}(R)\end{array}\right),
\end{equation*}
where $ R^{n-1} $ is considered as an $ (R,T_{n-1}(R))-$bimodule with the obvious scalar multiplications.
\end{remark}

The question given above suggests an analogue question on superbiderivations of $ \ZZ_{2}-$graded rings:
\\

\textbf{Question.} \textit{Under what conditions on a $ \ZZ_{2}-$graded ring, every Jordan super-biderivation is a Jordan biderivation?}
\\

Theorem \ref{2.6} answers the above question affirmatively for the trivial extension $ T(R,M) $, and hence for the upper triangular matrix ring $ T $ and $ T_{n}(R) $, without any additional condition.


\end{document}